\documentclass[11pt]{amsart}

\usepackage{amsmath, amssymb, amscd, cancel, graphicx, paralist, soul, stmaryrd}

\usepackage{mathdots}

\usepackage[all]{xy}
\usepackage{multirow}
\usepackage{longtable}
\usepackage{array}

\newtheorem{thm}{Theorem}

\newtheorem{lem}[thm]{Lemma}

\def\bE{{\mathbb E}}

\def\bP{{\mathbb P}}

\def\int{{\text{int}}}

\begin{document}

\title[On curves intersecting at most once]%
{On curves intersecting at most once}


\author[Joshua Evan Greene]{Joshua Evan Greene}

\address{Department of Mathematics, Boston College\\ Chestnut Hill, MA 02467}

\email{joshua.greene@bc.edu}


\maketitle






\medskip

\noindent {\bf Abstract.} 
We prove that on a closed surface of genus $g$, the cardinality of a set of simple closed curves in which any two are non-homotopic and intersect at most once is $\lesssim g^2 \log(g)$.
This bound matches the largest known constructions to within a logarithmic factor.
The proof uses a probabilistic argument in graph theory.
It generalizes as well to the case of curves that intersect at most $k$ times in pairs.

\vspace{.1in}


\section{Introduction.}

Let $S$ be a connected, oriented surface of finite type, and let $k$ be a non-negative integer.
A {\em $k$-system} of curves on $S$ is a collection of nonperipheral, essential, simple closed curves on $S$ such that any two are non-homotopic and intersect in at most $k$ points.
Juvan, Malni{\v c}, and Mohar raised the problem of estimating the size of the largest $k$-system on $S$ and proved that it is finite for any $S$ and $k$ \cite[Theorem 3.3]{jmm96}.
The case that $S = S_g$ is the closed surface of genus $g$ and $k=1$ has drawn particular interest, since it already demonstrates the difficulty of the problem: indeed, it remains unsolved.
Different authors have described constructions of 1-systems on $S_g$ whose size grows as a quadratic function of the genus $g$ \cite[Theorem 1.2]{aougab14},
\cite[Theorem 1]{mrt14}, but no larger constructions are known, and the best known upper bounds grow faster than a quadratic function of $g$.

P. Przytycki made dramatic progress on the problem by giving an exact answer to the corresponding problem for arcs.
In this variation on the problem, $S$ has punctures, and simple closed curves are replaced by simple arcs that limit to punctures at their ends.
Przytycki showed that if $S$ has Euler characteristic $\chi < 0$, then the maximum cardinality of a 1-system of arcs on $S$ is exactly $2|\chi|(|\chi|+1)$\cite[Theorem 1]{przytycki2015}.
The proof is an elegant argument in hyperbolic geometry.
He used this bound to show that if $\Gamma$ is a 1-system of curves on $S_g$, then $|\Gamma| \lesssim g^3$.
(For a function $B$ of several variables, the notation $A \lesssim B$ means that $A \le C \cdot B$ for some absolute constant $C$.)
The argument is inductive:
one selects any curve and argues that it meets $\lesssim |\chi|^2$ other curves, using the bound on arcs;
cutting the surface along the first curve and discarding $\lesssim |\chi|^2$ curves, one gets a 1-system on a surface of lower genus, and its size is bounded by induction \cite[Theorem 4]{przytycki2015}.
Aougab, Biringer, and Gaster improved the upper bound to $\lesssim g^3/(\log g)^2$ using deeper tools from hyperbolic geometry \cite[Theorem 1.1]{abg2017}; again, their argument relies on Przytycki's bound for arcs.

Our main result is an upper bound on the size of a 1-system of curves which comes to within a logarithmic factor of the order of growth of the largest known constructions:

\begin{thm}
\label{thm: main}
A 1-system of curves $\Gamma$ on $S_g$ has size $|\Gamma| \lesssim g^2 \log(g)$.
\end{thm}

The proof of Theorem \ref{thm: main} draws inspiration from the proof of \cite[Theorem 4]{przytycki2015} described above, but it takes on a different character.
We show that for any subset of curves $\Gamma' \subset \Gamma$, the number of curves in $\Gamma$ that intersect a unique curve in $\Gamma'$ is 
small (Lemma \ref{lem: bound}).
On the other hand, we prove the existence of a subset of curves $\Gamma'$ such that a large proportion of the curves in $\Gamma$ intersects a unique 
curve in $\Gamma'$ (Theorem \ref{thm: graph}).
The existence argument is probabilistic and graph theoretic in nature.
Together, Lemma \ref{lem: bound} and Theorem \ref{thm: graph} lead at once to the bound in Theorem \ref{thm: main}.

We adapt our method to bound the cardinality of a $k$-system of curves in Theorem \ref{thm: main 2}, applying an idea from \cite{abg2017}.
On the other hand, we present a contrast to Theorem \ref{thm: graph} in Theorem \ref{thm: graph 2} that indicates the limitation to our approach of bounding the size of a $k$-system of curves in terms of corresponding bounds on the size of a $k$-system of arcs.

\section{The main argument.}

We begin by preparing some graph theoretic notation.
Let $G = (V,E)$ denote a finite, simple graph.
For a vertex $v \in V$, let $d(v)$ denote its degree.
An isolated vertex is a vertex of degree 0.
Let $\Delta(G) = \max \{ d(v) \, | \,  v \in V \}$ denote the maximum vertex degree in $G$.
We require a generalization of the quantity $\Delta(G)$.
For a subset $V' \subset V$, let $U(V')$ denote the subset of vertices in $G$ with a unique neighbor in $V'$.
Thus, when $V'$ consists of a single vertex $v$, we have $|U(V')| = d(v)$.
Let $\Upsilon(G) = \max \{ |U(V')| \, | \, V' \subset V \}$.
Trivially, $\Upsilon(G) \ge \Delta(G)$.

For a finite collection of curves $\Gamma$ on $S_g$, let $G(\Gamma)$ denote the {\em intersection graph} of $\Gamma$:
this is the finite, simple graph with vertex set $V = \Gamma$ and edge set $E$ consisting of pairs of curves in $\Gamma$ that intersect non-trivially.

The following result generalizes the estimate $\Delta(G(\Gamma)) \lesssim g^2$ contained in the proof of \cite[Theorem 4]{przytycki2015}:

\begin{lem}
\label{lem: bound}
If $\Gamma$ is a 1-system on $S_g$, then $\Upsilon(G(\Gamma)) \lesssim g^2$.
\end{lem}

\begin{proof}
Let $\Gamma' \subset \Gamma$.
Resolve the crossings between the curves in $\Gamma'$ to obtain a set of pairwise disjoint curves on $S$.
Let $\Gamma''$ denote the subset of the resolved curves that contain a point of intersection with a curve in $U(\Gamma')$.
Thus, each curve in $\Gamma''$ meets some curve $\gamma \in U(\Gamma')$ in a single point, and $\gamma$ meets no other curve in $\Gamma''$.
It follows that cutting $S$ along the curves in $\Gamma''$ results in a {\em connected} surface $S'$ of Euler characteristic $\chi = 2-2g$.
Each curve in $U(\Gamma')$ cuts open to an arc on $S'$.
As in the proof of \cite[Theorem 4]{przytycki2015}, any two curves that cut open to the same homotopy type of arc on $S'$ are related by a Dehn twist about the unique curve in $\Gamma''$ that both intersect.
It follows that no more than two curves of $U(\Gamma')$ can cut open to the same homotopy type of arc on $S'$.
Discarding duplicates,  we thereby obtain a 1-system $A$ of at least $|U(\Gamma')|/2$ arcs on $S'$.
On the other hand, $|A| \le 2|\chi|(|\chi|+1) \lesssim g^2$ by \cite[Theorem 1]{przytycki2015}.
Therefore, $|U(\Gamma')| \lesssim g^2$, leading to the desired bound.
\end{proof}


By contrast, the following estimate on $\Upsilon(G)$ holds for an arbitrary simple graph $G$:

\begin{thm}
\label{thm: graph}
If $G=(V,E)$ is a simple graph with $n$ vertices, no isolated vertices, and maximum degree $\Delta$, then $\Upsilon(G) \gtrsim n / \log \Delta$.
\end{thm}


\begin{proof}
Form a subset $V' \subset V$ by selecting each vertex from $V$ at random, independently, with probability $p$ apiece.
The expected number of neighbors that $v \in V$ has in $V'$ is $p \cdot d(v)$, which is $\sim 1$ if $p \sim 1 / d(v)$.
Therefore, if we tune the parameter $p$ to equal $1/d$, where $d(v) \sim d$ for many $v \in V$, then many of these vertices will have a unique neighbor in $V_0$.

We now rigorize this heuristic.
The random variable $|U(V')|$ is the sum of random variables $U_v$, $v \in V$, where each $U_v$ denotes the indicator random variable of the event $A_v$ that $v$ has a unique neighbor in $V'$.
Let $\bE(\cdot)$ denote the expectation of a random variable and $\bP(\cdot)$ the probability of an event.
We have
\[
\bE(U_v) = \bP(A_v) = d(v) p (1-p)^{d(v)-1},
\]
so by linearity of expectation, we have
\[
\bE(|U(V')|) = \sum_{v \in V} d(v) p (1-p)^{d(v)-1}.
\]

We wish to select $p$ so as to make $\bE(|U(V')|)$ large.
We do by an application of the dyadic pigeonhole principle.
Let $V_j = \{ v \in V \, | \, 2^{j-1} \le \deg(v) \le 2^j \}$, $j=1,\dots, \lfloor \log_2 \Delta \rfloor$.
As there are no isolated vertices in $G$, the union of the sets $V_j$ is all of $V$, so there exists a value $j$ for which $|V_j| \ge n / \log_2 \Delta$.
Set $p = 1 / 2^j$.
Using the estimate $(1-1/x)^x \ge 1/4$, valid for all $x \ge 2$, it follows that $d p (1-p)^{d-1} \ge \frac 12 dp 4^{-dp}$ for all $d \ge 1$.
Moreover, the function $ \frac 12 y 4^{-y}$ is concave-down for $y \in [1/2,1]$, so it is bounded below on this interval by the value that it takes at endpoints, which is $1/8$.
Consequently, at least $n / \log_2 \Delta$ of the terms in the summation for $\bE(|U(V')|)$ are $\ge 1/8$, so $\bE(|U(V')|) \ge n / (8 \log_2 \Delta)$.
Thus, there exists a subset $V' \subset V$ for which $|U(V')| \ge n / ( 8 \log_2 \Delta)$, as desired.
\end{proof}

\begin{proof}
[Proof of Theorem \ref{thm: main}]
If $\Gamma$ contains a curve disjoint from the rest, then the result follows by an easy induction on the genus $g$.
Otherwise, apply Lemma \ref{lem: bound} and Theorem \ref{thm: graph} to $G(\Gamma)$ to obtain $|\Gamma| / \log g \lesssim n / \log \Delta(G(\Gamma))$ $\lesssim \Upsilon(G(\Gamma)) \lesssim g^2.$
\end{proof}


\section{Extending the argument.}

Theorem \ref{thm: main} admits a straightforward generalization to $k$-systems:

\begin{thm}
\label{thm: main 2}
If $\Gamma$ is a $k$-system of curves on $S_g$, then $|\Gamma| \lesssim_k  g^{3k-1} \log (g)$. \qed
\end{thm}

(For a function $B$ of several variables including $k$, the notation $A \lesssim_k B$ means that $A \le C(k) \cdot B$ for some function $C(k)$ of $k$ alone.)
By contrast, the largest known construction of a $k$-system of curves on $S_g$ for even values $k$ has size $\gtrsim_k g^{k+1}$ \cite[Remark after Theorem 1.2]{abg2017}.
The proof of Theorem \ref{thm: main 2} is identical to that of Theorem \ref{thm: main}, using the following generalization of Lemma \ref{lem: bound} and \cite[Theorem 1.4]{abg2017}:


\begin{lem}
\label{lem: bound 2}
If $\Gamma$ is a $k$-system on $S_g$, then $\Upsilon(G(\Gamma)) \lesssim_k g^{3k-2}$.
\end{lem}


\begin{proof}
Let $\Gamma' \subset \Gamma$.
Let $C' \subset S_g$ denote the union of the curves in $\Gamma'$.
Select one curve in $\Gamma'$ from each component of $C'$.
These curves are pairwise disjoint, and no two are isotopic.
Therefore, the number of these curves, and hence the number of components of $C'$, is $\lesssim g$.
By induction on the number of intersection points between the curves in $\Gamma'$, we can resolve the intersection points of $C'$ to obtain a collection $\Gamma''$ of simple closed curves on $S_g$ with $|\Gamma''| = |\pi_0(C')| \lesssim g$.
Orient each curve $\gamma'' \in \Gamma''$ and place a point $p(\gamma'')$ on it disjoint from the rest of $\Gamma$.
Each curve $\gamma \in U(\Gamma')$ meets a unique curve $\gamma'' \in \Gamma''$.
Locate the unique point of intersection $p \in \gamma \cap \gamma''$ with the property that the oriented arc $\alpha \subset \gamma''$ from $p$ to $p(\gamma'')$ is disjoint from $\gamma$.
Slide $\gamma$ along $a$ so as to produce an arc with endpoints at $p(\gamma'')$.
Doing so for each $\gamma \in U(\Gamma)$ results in a set $A$ of $|U(\Gamma)|$ arcs on the surface $S'$ obtained by puncturing $S_g$ at all of the points $p(\gamma'')$, $\gamma'' \in \Gamma''$.
Since there are $\lesssim g$ punctures, we have $|\chi(S')| \lesssim g$.
As in the proof of \cite[Theorem 4.1]{abg2017}, $A$ forms a $(3k-2)$-system on $S'$.
By \cite[Theorem 1.5]{przytycki2015}, $|U(\Gamma')| = |A| \lesssim_k |\chi(S')|^{3k-1} \lesssim g^{3k-1}$.
The resulting bound on $\Upsilon(G(\Gamma))$ now follows.
\end{proof}

If one knew that $\Upsilon(G(\Gamma)) \lesssim |\Gamma|$ for any 1-system $\Gamma$, then the proof of Theorem \ref{thm: main} would yield the estimate $|\Gamma| \lesssim g^2$.
The following result indicates that this is {\em not} the case for arbitrary finite, simple graphs:

\begin{thm}
\label{thm: graph 2}
For all $n, \Delta > 0$ such that $n \gtrsim \Delta \log \Delta$, there exists a simple graph $G=(V,E)$ with $\sim n$ vertices, no isolated vertices, and maximum degree $\sim \Delta$ such that $\Upsilon(G) \lesssim n \log \log \Delta / \log \Delta$.
\end{thm}

Here the notation $\sim$ means equality to within a factor of 2.
The construction of the examples of Theorem \ref{thm: graph 2} is inspired by the proof of Theorem \ref{thm: graph}.
We search for graphs whose vertex degrees are equidistributed on a logarithmic scale, so that all of the sets $V_j$ appearing in that proof have the same cardinality.
The heuristic for constructing the graphs is to take as $V_j$ a set of $t$ vertices, each connected to $2^j$ half-edges, for $j=1,\dots,k$.
We then join the ends of these half-edges at random into edges to form a graph $G$.
I thank Larry Guth for suggesting this construction and sketching why $\Upsilon(G)$ should behave like $n / \log \Delta$.
The actual construction we describe in the proof is based on the existence of a family of expander graphs.
We suspect that with care, the factor of $ \log \log \Delta$ can be removed from the statement of Theorem \ref{thm: graph 2}.


\begin{proof}
Set $k = \frac 12 \lfloor \log_2 \Delta \rfloor$, 
 and let $t$ denote the smallest power of 2 that is greater than or equal to $n/2k$.
Let $V_1,\dots,V_k$ denote disjoint sets of cardinality $t$ apiece.
For each pair of distinct indices $1 \le i , j \le k$, let $G_{ij}$ denote a bipartite Ramanujan graph with bipartition $V_i \sqcup V_j$ in which each vertex has degree $d_{ij} = 2^{i+j}$.
For each index $1 \le i \le k$, let $G_{ii}$ denote a bipartite Ramanujan graph on the vertex set $V_i$ in which each vertex has degree $d_{ii} = 2^{2i}$.
The existence of such graphs follows from \cite[Theorem 5.5]{mss2015}.
The only condition we need to ensure is that $t \ge d_{ij}$ for all $i,j$, and this follows on the assumption that $n \ge \Delta \log_2 \Delta$.
The key feature of $G_{ij}$ that we shall require is that for every pair of subsets $A \subset V_i$ and $B \subset V_j$, we have
\begin{equation}
\label{eq: expander}
| |E(A,B)| - d_{ij} |A| |B| / t | \le 2 (d_{ij} |A| |B|)^{1/2}.
\end{equation}
Here $E(A,B) \subset E(G_{ij})$ denotes the subset of edges with one endpoint in $A$ and one endpoint in $B$.
Inequality \eqref{eq: expander} follows by inserting the defining condition on the second eigenvalue of a Ramanujan graph into \cite[Corollary 9.2.5]{asbook}.
Finally, let $G$ be the union of all of the graphs $G_{ij}$ on the vertex set $V = V_1 \sqcup \cdots \sqcup V_k$.
Observe that $G$ has $kt \sim n$ vertices and maximum degree $2^{k+1} + 2^{k+2} + \cdots + 2^{k + k} \sim \Delta$.

We now argue that $\Upsilon(G) \lesssim n \log \log \Delta / \log \Delta$.
Consider an arbitrary subset $V' \subset V$.
Write
\[
U = U(V'), \quad U_j = U \cap V_j,  \quad V_j' = V' \cap V_j, \quad j=1,\dots,k.
\]
We wish to bound $|U| = \sum_{j=1}^k |U_j|$.
We do so by breaking the sum into three parts, as follows.
Let $j_0$ denote the smallest index, if it exists, with the property that $|E(V',V_{j_0})| \ge t$, and set $j_0 = \infty$ otherwise.
By construction, $|E(V',V_{j+1})| = 2 |E(V',V_j)|$ for all $j$.
The union of the $V_j$ with $j \ll j_0$ will contain few neighbors of $V'$ and so few ($\lesssim t$) elements of $U$.
The union of the $V_j$ with $j \approx j_0$ may contain many ($\lesssim t \log k$) elements of $U$.
Finally, the expansion property \eqref{eq: expander} will show that the union of $V_j$ for $j \gg j_0$ will contain few ($\lesssim t$) elements of $U$.
Together, these bounds lead to the desired bound on $|U|$.

We now rigorize this heuristic.
First, we have
\begin{equation}
\label{eq: u bound 1}
\sum_{j < j_0} |U_j| = \sum_{j < j_0} |E(V',U_j)| \le \sum_{j < j_0} |E(V',V_j)| < (\cdots + \frac 14 + \frac 12 + 1) t \lesssim t.
\end{equation}
In particular, if $j_0 = \infty$, then we reach the desirable conclusion that $|U| \lesssim t \lesssim n / \log \Delta$.
Thus, we may assume that $j_0$ is finite.

By the pigeonhole principle, there exists an index $i$ with the property that $|E(V_i',V_{j_0})| \ge t/k$.
For this index $i$ and for all $j$, we have
\begin{equation}
\label{eq: geometric}
d_{ij} |V_i'| = |E(V_i',V_j)| = 2^{j-j_0} |E(V_i',V_{j_0})|  \ge 2^{j-j_0} t / k.
\end{equation}

Suppose that $2 (d_{ij} |V_i'||U_j|)^{1/2} < \frac 12 d_{ij} |V_i'| |U_j| / t$ for some value $j \ge j_0$.
By a trivial estimate, the expansion property \eqref{eq: expander}, and \eqref{eq: geometric}, we have
\[
|U_j| \ge |E(V_i',U_j)| > \frac 12 d_{ij} |V_i'| |U_j| / t \ge 2^{j-j_0 -1} |U_j| / k.
\]
We thereby obtain $k > 2^{j-j_0 -1}$, so $j < j_0 + \log_2 k + 1$ in this case.
We apply the trivial bound
\begin{equation}
\label{eq: u bound 2}
\sum_{j=j_0}^{j_0 + \log_2 k} |U_j| \lesssim t \log k
\end{equation}
for these values $j$.

If instead $j \ge j_0 + \log_2 k + 1$, then it follows that $2 (d_{ij} |V_i'||U_j|)^{1/2} \ge \frac 12 d_{ij} |V_i'| |U_j| / t$.
Rearranging this inequality and invoking \eqref{eq: geometric} once more, we obtain $|U_j| < 16 t^2 / (d_{ij} |V_i'|) \lesssim tk 2^{j_0-j}$.
Consequently,
\begin{equation}
\label{eq: u bound 3}
\sum_{j \ge j_0 + \log_2 k + 1} |U_j| \lesssim t k \sum_{j \ge j_0 + \log_2 k + 1} 2^{j_0-j} \lesssim t.
\end{equation}

Combining \eqref{eq: u bound 1},  \eqref{eq: u bound 2}, and \eqref{eq: u bound 3}, we obtain $|U| \lesssim t \log k \lesssim n \log \log \Delta / \log \Delta$, as desired.
\end{proof}

\section*{Acknowledgements}

I thank Ravi Boppana, Jonah Gaster, and Larry Guth for fun and helpful conversations.
In particular, Larry laid the groundwork for constructing the ``enemy graphs" of Theorem \ref{thm: graph 2}.
This work was supported by NSF CAREER Award DMS-1455132.

\bibliographystyle{myalpha}
\bibliography{/Users/JoshuaGreene/Dropbox/Papers/References}

\end{document}